\documentclass[10pt]{article}

\usepackage{amstext}
\usepackage{amssymb}
\usepackage{amsmath}
\usepackage{amsthm}
\usepackage[dvips]{graphicx}        
\graphicspath{{./pict/}}            

\theoremstyle{plain}
\newtheorem{theorem}{Theorem}[section]
\newtheorem{lemma}[theorem]{Lemma}

\newtheorem{corol}[theorem]{Corollary}

\theoremstyle{definition}
\newtheorem{definition}[theorem]{Definition}

\newtheorem{remark}[theorem]{Remark}

\begin{document}
\vspace{\baselineskip}

\vspace{\baselineskip} \thispagestyle{empty}

\title{Hurwitz 
rational functions}

\author{Yury Barkovsky
   \\
 \small Department of calculus mathematics and mathematical physics,\\ \small Faculty of Mathematics, Mechanics \& Computer Science,\\
 \small  Southern Federal University, \\
 \small Milchakova str. 8a, 344090, Rostov-on-Don, Russia
 \and  Mikhail Tyaglov\thanks{The work of M.T. was supported by the
Sofja Kovalevskaja Research Prize of Alexander von Humboldt Foundation.
Email: {\tt tyaglov@math.tu-berlin.de}}\\
\small Technische Universit\"at  Berlin, Institut f\"ur Mathematik,\\
\small MA 4-5, Strasse des 17. Juni 136, 10623, Berlin, Germany}

\date{\small \today}


\maketitle

\vspace{10mm}

\begin{abstract}
A generalization of Hurwitz stable polynomials to real rational functions is considered. We establishe
an analogue of the Hurwitz stability criterion for rational functions and introduce a new type
of determinants that can be treated as a generalization of the Hurwitz determinants.
\end{abstract}


%

%



\section*{Introduction}

\hspace{4mm} It is well known that the problem of stability of a linear difference or differential
system with constant coefficients reduces to the question of locating the zeroes of its
characteristic polynomial in the left half-plane of the complex plane. One of the most famous
results from stability theory is the Hurwitz theorem, which expresses stability of a real polynomial
in terms of its coefficients~\cite{Hurwitz,{KreinNaimark},{Gantmakher}} (see also~\cite{Barkovsky.2,{Tyaglov.general.Hurw}}).
Namely, the Hurwitz theorem states that a real polynomial has all its zeroes in the open
left half-plane if and only if some determinants constructed with the coefficients of the polynomial
are positive (see Theorem~\ref{Theorem.Hurwitz.stable.Hurwitz.matrix.criteria}). Those determinants are
now called the Hurwitz determinants due to Adolf Hurwitz who introduced them in~\cite{Hurwitz}.

In the present work we investigate a class of real rational functions satisfying the Hurwitz conditions:
the Hurwitz determinants constructed with the coefficients of the Laurent series at $\infty$ of a real rational function
are positive up to the order $n$, where $n$ is the sum of degrees of the numerator and denominator of the rational
function. This is a generalization of the class of real Hurwitz stable polynomials. It turns out that such class
of rational functions is characterized by location of poles and zeroes: all zeroes lie in the open left
half-plane of the complex plane while all poles lie in the open right half-plane (Theorem~\ref{Hurvitz.analog}).
Finally, we express the Hurwitz determinants of real rational functions in terms of the coefficients of the numerator
and denominator (Lemma~\ref{lem.rat.1}). Thereby we introduce a new type of determinants and describe the class of real rational functions
satisfying Hurwitz conditions in terms of coefficients of their numerator and denominator (Theorem~\ref{Hurvitz.analog.2}). As well
as in case of polynomials this class of rational functions is characterized by positivity of those determinants. Some simple
properties of the new type of determinants that can be treated as a generalization of the Hurwitz determinants are considered.


\setcounter{equation}{0}

\section{Hurwitz polynomials}\label{section:Hurwitz.poly}

\hspace{4mm} Consider a real polynomial
\begin{equation}\label{main.polynomial}
p(z)\stackrel{def}{=}a_0z^n+a_1z^{n-1}+\dots+a_n,\qquad a_1,\dots,a_n\in\mathbb
R,\ a_0>0.
\end{equation}
Throughout the paper we use the following notation
\begin{equation}\label{floor.poly.degree}
l\stackrel{def}{=}\left[\dfrac n2\right],
\end{equation}
where $n=\deg p$, and $[\rho]$ denotes the largest integer not
exceeding $\rho$.

The polynomial $p$ can always be represented as follows
\begin{equation*}\label{app.poly.odd.even}
p(z)=p_0(z^2)+zp_1(z^2),
\end{equation*}
where $p_0$ and $p_1$ are the even and odd parts of the polynomial, respectively.
Introduce the following function:
\begin{equation}\label{assoc.function}
\Phi(u)\stackrel{def}{=}\displaystyle\frac {p_1(u)}{p_0(u)}.
\end{equation}
\begin{definition}\label{def.associated.function}
We call $\Phi$ the \textit{function  associated with the
polynomial}~$p$.
\end{definition}

Associate with the polynomial $p$ the following determinants:
\begin{equation}\label{delta}
\Delta_{j}(p)\stackrel{def}{=}
\begin{vmatrix}
a_1&a_3&a_5&a_7&\dots&a_{2j-1}\\
a_0&a_2&a_4&a_6&\dots&a_{2j-2}\\
0  &a_1&a_3&a_5&\dots&a_{2j-3}\\
0  &a_0&a_2&a_4&\dots&a_{2j-4}\\
\vdots&\vdots&\vdots&\vdots&\ddots&\vdots\\
0  &0  &0  &0  &\dots&a_{j}
\end{vmatrix},\quad j=1,\ldots,n,
\end{equation}
where we set $a_i\equiv0$ for $i>n$.
\begin{definition}\label{def.Hurwitz.dets}
The determinants $\Delta_{j}(p)$, $j=1,\ldots,n$, are called the
\textit{Hurwitz determinants} or the \textit{Hurwitz minors} of
the polynomial~$p$.
\end{definition}

Suppose that $\deg p_0\geqslant\deg p_1$ and expand the function
$\Phi$ into its Laurent series at~$\infty$:
\begin{equation}\label{app.assoc.function.series}
\Phi(u)=\dfrac{p_1(u)}{p_0(u)}=s_{-1}+\frac{s_0}u+\frac{s_1}{u^2}+\frac{s_2}{u^3}+\frac{s_3}{u^4}+\dots,
\end{equation}
where $s_{-1}\neq0$ if $\deg p_0=\deg p_1$, and $s_{-1}=0$ if
$\deg p_0>\deg p_1$.

For a given infinite sequence $(s_j)_{j=0}^{\infty}$, consider the determinants
\begin{equation}\label{Hankel.determinants.1}
D_j(\Phi)\stackrel{def}{=}
\begin{vmatrix}
    s_0 &s_1 &s_2 &\dots &s_{j-1}\\
    s_1 &s_2 &s_3 &\dots &s_j\\
    \vdots&\vdots&\vdots&\ddots&\vdots\\
    s_{j-1} &s_j &s_{j+1} &\dots &s_{2j-2}
\end{vmatrix},\quad j=1,2,3,\dots,
\end{equation}
These determinants are referred to as the~\textit{Hankel minors} or~\textit{Hankel determinants}.

Together with the determinants~\eqref{Hankel.determinants.1} we consider one more sequence of Hankel determinants.
\begin{equation}\label{Hankel.determinants.2}
\widehat{D}_j(\Phi)\stackrel{def}{=}
\begin{vmatrix}
    s_1 &s_2 &s_3 &\dots &s_{j}\\
    s_2 &s_3 &s_4 &\dots &s_{j+1}\\
    \vdots&\vdots&\vdots&\ddots&\vdots\\
    s_{j} & s_{j+1} & s_{j+2} & \dots &s_{2j-1}
\end{vmatrix},\quad j=1,2,3,\dots
\end{equation}

It is very well known~\cite{Hurwitz,Gantmakher} (see also~\cite{Tyaglov.general.Hurw}) that
there are relations between the determinants $D_j(\Phi)$, $\widehat{D}_j(\Phi)$
and the Hurwitz minors~$\Delta_j(p)$:
\begin{itemize}
\item[1)]
If $n=2l$, then
\begin{equation}\label{Formulae.Gurwitz.1}
\begin{array}{l}
\Delta_{2j-1}(p)=a_0^{2j-1}D_j(\Phi),\\
\\
\Delta_{2j}(p)=(-1)^ja_0^{2j}\widehat{D}_j(\Phi),
\end{array}\qquad
j=1,2,\dots,l;\qquad\quad
\end{equation}
\item[2)]
If $n=2l+1$, then
\begin{equation}\label{Formulae.Gurwitz.2}
\begin{array}{l}
\Delta_{2j}(p)={\left(\dfrac{a_0}{s_{-1}}\right)}^{2j}D_j(\Phi),\\
\Delta_{2j+1}(p)=(-1)^j{\left(\dfrac{a_0}{s_{-1}}\right)}^{2j+1}\widehat{D}_j(\Phi),
\end{array}\quad j=0,1,\dots,l;
\end{equation}
where $\widehat{D}_0(\Phi)\equiv1$.
\end{itemize}
Here $l$ is defined in~\eqref{floor.poly.degree}.

It is also well known~\cite{Hurwitz,Gantmakher} (see also~\cite{Holtz_Tyaglov}) that the number of poles of
the function $\Phi$ equals the order of the last non-zero minor $D_j(\Phi)$. Since the function
$\Phi$ has at most $l$ poles, we have
\begin{equation}\label{Theorem.Hurwitz.stable.poly.and.sign.regularity.condition.3}
D_j(\Phi)=\widehat{D}_j(\Phi)=0,\quad j>l.
\end{equation}
Thus, in the sequel, we deal only with the determinants $D_j(\Phi)$,
$\widehat{D}_j(\Phi)$ of order at most $l$.

\begin{definition}\label{def.Hurwitz.stable.poly}
The polynomial $p$ defined in~\eqref{main.polynomial} is called
\textit{Hurwitz} or \textit{Hurwitz stable} if all its zeroes lie
in the \textit{open} left half-plane of the complex plane.
\end{definition}

The following criterion of Hurwitz stability of a real polynomial was (implicitly) established
in~\cite{Hurwitz} (see also~\cite{Gantmakher,{Barkovsky.2},{Tyaglov.general.Hurw}}).
\begin{theorem}\label{Theorem.Hurwitz.stable.poly.and.sign.regularity}
Let a real polynomial $p$ be defined by~\eqref{main.polynomial}.
The following conditions are equivalent:
\begin{itemize}
\item[1)] the polynomial $p$ is Hurwitz stable;
\item[2)] the following hold
\begin{equation}\label{Theorem.Hurwitz.stable.poly.and.sign.regularity.condition.0}
\begin{split}
&s_{-1}>0\quad\text{for}\quad n=2l+1,\\
&D_j(\Phi)>0, \qquad j=1,\ldots,l,\\
&(-1)^{j}\widehat{D}_{j}(\Phi)>0, \qquad j=1,\ldots,l,
\end{split}
\end{equation}
where $l=\left[\dfrac n2\right]$.
\end{itemize}
\end{theorem}

This theorem together with fornl\ae~\eqref{Formulae.Gurwitz.1}--\eqref{Formulae.Gurwitz.2} imply the following
theorem which is very well known as the Hurwitz criterion of polynomial stability.
\begin{theorem}[Hurwutz~\cite{Hurwitz}]\label{Theorem.Hurwitz.stable.Hurwitz.matrix.criteria}
A real polynomial $p$ of degree $n$ as
in~\eqref{main.polynomial} is Hurwitz stable if and only if all Hurwitz determinants $\Delta_j(p)$ are positive:
\begin{equation}\label{Hurvitz.det.noneq}
\Delta_1(p)>0,\ \Delta_2(p)>0,\dots,\ \Delta_n(p)>0;
\end{equation}
\end{theorem}

\setcounter{equation}{0}

\section{Hurwitz rational function}\label{s:stabe.rat.functions}

\hspace{4mm} Consider a rational function
\begin{equation}\label{rat.func.1}
R(z)\stackrel{def}{=}\dfrac{h(z)}{g(z)}=t_0z^{r-m}+t_1z^{r-m-1}+t_2z^{r-m-2}+\ldots,
\end{equation}
where $h$ and $g$ are real polynomials
\begin{equation}\label{rat.func.2}
\begin{split}
&h(z)\stackrel{def}{=} b_0z^r+b_1z^{r-1}+\ldots+b_{r-1}z+b_r,\qquad
b_0,b_1,\dots,b_r\in\mathbb R,\ b_0>0;\\
&g(z)\stackrel{def}{=} c_0z^m+c_1z^{m-1}+\ldots+c_{m-1}z+c_m,\qquad
c_0,c_1,\dots,c_m\in\mathbb R,\ c_0>0.
\end{split}
\end{equation}
Here $r,m\in\mathbb{N}\bigcup\{0\}$, $r+m>0$. Assume that
the polynomials $p$ and $q$ are coprime. The number $n=r+m$ is called
the \textit{order} of the function $R$.
\begin{definition}
The real rational function $R$ is called \textit{Hurwitz} rational function
if both polynomials $h(z)$ and $g(-z)$ are Hurwitz stable\footnote{Or one of them
is Hurwitz stable while the second one is a constant.}.
\end{definition}

It turns out that one can establish a criterion for real rational functions to be Hurwitz
similar to the Hurwitz criterion of polynomial stability
\begin{theorem}\label{Hurvitz.analog}
A real rational function~$R(z)$ of the form~\eqref{rat.func.1}--\eqref{rat.func.2} is Hurwitz
if and only if the leading principal minors $\Delta_k(R)$ of the infinite Hurwitz matrix
\begin{equation}\label{Hurwitz.matrix}
\mathcal{H}(R)\stackrel{def}{=}
\begin{pmatrix}
t_1&t_3&t_5&t_7&t_9&\dots\\
t_0&t_2&t_4&t_6&t_8&\dots\\
0  &t_1&t_3&t_5&t_7&\dots\\
0  &t_0&t_2&t_4&t_6&\dots\\
0  &0  &t_1&t_3&t_5&\dots\\
0  &0  &t_0&t_2&t_4&\dots\\
\vdots&\vdots&\vdots&\vdots&\vdots&\ddots\\
\end{pmatrix}
\end{equation}
are positive up to the order $n$ $($inclusive$)$:
\begin{equation*}
\Delta_1(R)>0,\ \Delta_2(R)>0,\dots,\ \Delta_n(R)>0.
\end{equation*}
\end{theorem}
\begin{proof} Consider the following auxiliary  polynomial
\begin{equation}\label{rat.func.3}
P(z)\stackrel{def}{=}(-1)^mh(z)g(-z)=a_0z^{n}+a_1z^{n-1}+\ldots+a_{n-1}z+a_n,\quad
a_0=b_0c_0>0.
\end{equation}
By definition, the function $R$ is Hurwitz if and only if the polynomial $P$ is Hurwitz stable.

Consider the rational function $\Phi$ associated with the polynomial $P$:
\begin{equation}\label{rat.func.4}
\Phi(u)=\dfrac{P_1(u)}{P_0(u)}=
s_{-1}+\frac{s_0}u+\frac{s_1}{u^2}+\frac{s_2}{u^3}+\dots,
\end{equation}
where the polynomials $P_0(u)$ ш $P_1(u)$ are the even and odd parts of the polynomial $P$, respectively.
Since $2P_0(z^2)=P(z)+P(-z)$ and $2zP_0(z^2)=P(z)-P(-z)$, we have
\begin{equation*}
z\Phi(z^2)=\dfrac{P(z)-P(-z)}{P(z)+P(-z)}=
\dfrac{h(z)g(-z)-h(-z)g(z)}{h(z)g(-z)+h(-z)g(z)}=\dfrac{R(z)-R(-z)}{R(z)+R(-z)}.
\end{equation*}
Hence we obtain
\begin{equation}\label{rat.func.6}
\Phi(u)=\dfrac{R_1(u)}{R_0(u)}=\dfrac{P_1(u)}{P_0(u)},
\end{equation}
where
\begin{equation*}
R_0(z^2)=\dfrac{R(z)+R(-z)}{2},
\end{equation*}
\begin{equation*}
R_1(z^2)=\dfrac{R(z)-R(-z)}{2z}
\end{equation*}
are the even and odd parts of the function $R$, respectively. Thus, for $n=2l+1$, we have
\begin{equation*}
\Phi(u)=\dfrac{t_0+t_2u^{-1}+t_4u^{-2}+\ldots}{t_1+t_3u^{-1}+t_5u^{-2}+\ldots}=
s_{-1}+\frac{s_0}u+\frac{s_1}{u^2}+\frac{s_2}{u^3}+\dots,
\end{equation*}
and for $n=2l$, we have
\begin{equation*}
\Phi(u)=\dfrac{t_1u^{-1}+t_3u^{-2}+t_5u^{-3}+\ldots}{t_0+t_2u^{-1}+t_4u^{-2}+\ldots}=
s_{-1}+\frac{s_0}u+\frac{s_1}{u^2}+\frac{s_2}{u^3}+\dots
\end{equation*}
We recall that $n=r+m$ is the order of the rational function $R$. Now note that the
formul\ae~\eqref{Formulae.Gurwitz.1}--\eqref{Formulae.Gurwitz.2} are formal,
so following verbatim the proof of the formul\ae~\eqref{Formulae.Gurwitz.1}--\eqref{Formulae.Gurwitz.2}
(see, for instance,~\cite{Tyaglov.general.Hurw,{Holtz_Tyaglov}}) one can establish the following relationships
\begin{itemize}
\item[1)] for $n=2l$,
\begin{equation}\label{Formulae.Gurwitz.rat.1}
\Delta_{2j-1}(R)=t_0^{2j-1}D_j(\Phi),\
\Delta_{2j}(R)=t_0^{2j}\widehat{D}_j(\Phi)\qquad j=1,2,\dots,l,
\end{equation}
\item[2)] for $n=2l+1$,
\begin{equation}\label{Formulae.Gurwitz.rat.2}
\begin{split}
&\Delta_{2j}(R)={\left(\frac{t_0}{s_{-1}}\right)}^{2j}D_j(\Phi),\qquad j=1,2,\dots,l,\\
&\Delta_{2j+1}(R)=(-1)^j{\left(\frac{t_0}{s_{-1}}\right)}^{2j+1}\widehat{D}_j(\Phi),\qquad j=0,1,\dots,l.
\end{split}
\end{equation}
\end{itemize}
where the Hurwitz minors $\Delta_{j}(R)$ are defined as follows
\begin{equation}\label{Gurwitz.minors.rat}
\Delta_{j}(R)\stackrel{def}{=}
\begin{vmatrix}
t_1&t_3&t_5&t_7&\dots&t_{2j-1}\\
t_0&t_2&t_4&t_6&\dots&t_{2j-2}\\
0  &t_1&t_3&t_5&\dots&t_{2j-3}\\
0  &t_0&t_2&t_4&\dots&t_{2j-4}\\
\vdots&\vdots&\vdots&\vdots&\ddots&\vdots\\
0  &0  &0  &0  &\dots&t_{j}
\end{vmatrix}
\end{equation}

By Theorem~\ref{Theorem.Hurwitz.stable.poly.and.sign.regularity}, the polynomial $P$ is Hurwitz
stable if and only if the minors $D_j(\Phi)$ and
$\widehat{D}_j(\Phi)$ satisfy the inequalities~\eqref{Theorem.Hurwitz.stable.poly.and.sign.regularity.condition.0}.
Now since the polynomial $P$ is Hurwitz stable if and only if $R$ is a Hurwitz rational function,
the formul\ae~\eqref{Formulae.Gurwitz.rat.1}--\/\eqref{Formulae.Gurwitz.rat.2} and Theorem~\ref{Theorem.Hurwitz.stable.poly.and.sign.regularity}
imply the assertion of the theorem.
\end{proof}

Note that the leading principal minors $\Delta_j(R)$ of order
more than $n$ equal zero regardless whether the function $R$ being
Hurwitz or not:
\begin{equation}\label{zeros}
\Delta_{n+1}(R)=\Delta_{n+2}(R)=\Delta_{n+3}(R)=\cdots=0.
\end{equation}
It follows from~\eqref{Theorem.Hurwitz.stable.poly.and.sign.regularity.condition.3} and from the
folrmul\ae~\eqref{Formulae.Gurwitz.rat.1}--\/\eqref{Formulae.Gurwitz.rat.2} which are obviously valid
for $j>l$.

\begin{remark}
It is easy to see that $a_0=b_0c_0$ and $t_0=\dfrac{b_0}{c_0}$. So the formul\ae~\eqref{Formulae.Gurwitz.1}--\/\eqref{Formulae.Gurwitz.2} and
the formul\ae~\eqref{Formulae.Gurwitz.rat.1}--\/\eqref{Formulae.Gurwitz.rat.2} imply the equalities
\begin{equation}\label{Hurvitz.dets}
\Delta_j(P)=c_0^{2j}\Delta_j(R),\qquad j=1,2,\dots
\end{equation}
This equalities verify~\eqref{zeros}, since $\Delta_j(P)=0$ for $j>n$.
\end{remark}

Recall now that one of necessary conditions for the polynomial $p$ defined in~\eqref{main.polynomial}
to be Hurwitz stable is the positivity of its coefficients $a_j>0$, $j=0,1,\ldots,n$. This is the so-called
Stodola theorem~\cite{Hurwitz,{Gantmakher}}. It turns out that positivity of the coefficients $t_j$ in~\eqref{rat.func.1} is not necessary condition for
the Hurwitzness of the function $R$. Indeed, if $R=\dfrac{h}{g}$ is Hurwitz, then the polynomial $g$
has all zeroes in the open right half-plane, so it has coefficients of different signs. Therefore, it is easy
to find polynomials $h$ and $g$ such that $R$ has both negative and positive coefficients.
For example, the following function is Hurwitz, but its Laurent series at
infinity has positive, negative and zero coefficients:
\begin{equation*}
F(z)=\frac{z^2+z+1}{z^2-z+1}=t_0+t_1z^{-1}+t_2z^{-2}+t_3z^3+\dots,
\end{equation*}
where
\begin{equation*}
t_0=1\qquad\text{and}\qquad t_{3j-2}=t_{3j-1}=(-1)^{(j-1)}2,\quad  t_{3j}=0\qquad\text{for}\qquad j=1,2,\ldots
\end{equation*}

\vspace{3mm}

Finally, we find a connection between the coefficients of the polynomials $h$ and $g$ and the
Hurwitz determinants $\Delta_j(R)$. This gives us a criterion of Hurwitzness of a real rational function
in terms of the coefficients of its numerator and denominator.

Let again the function $R$ be defined by~\eqref{rat.func.1}--\/\eqref{rat.func.2}. Introduce the following determinants
of order $2j$, $j=1,2,\ldots$:
\begin{equation}\label{nabla}
\Omega_{2j}(h,g)\stackrel{def}{=}
        \begin{vmatrix}
                c_0&c_1&c_2&c_3&c_4&c_5&\dots&c_{j-1}&  c_j  &\dots&c_{2j-3}&c_{2j-2}&c_{2j-1}\\
                b_0&b_1&b_2&b_3&b_4&b_5&\dots&b_{j-1}&  b_j  &\dots&b_{2j-3}&b_{2j-2}&b_{2j-1}\\
                 0 & 0 &c_0&c_1&c_2&c_3&\dots&c_{j-3}&c_{j-2}&\dots&c_{2j-5}&c_{2j-4}&c_{2j-3}\\
                 0 &b_0&b_1&b_2&b_3&b_4&\dots&b_{j-2}&b_{j-1}&\dots&b_{2j-4}&b_{2j-3}&b_{2j-2}\\
                 0 & 0 & 0 & 0 &c_0&c_1&\dots&c_{j-5}&c_{j-4}&\dots&c_{2j-7}&c_{2j-6}&c_{2j-5}\\
                 0 & 0 &b_0&b_1&b_2&b_3&\dots&b_{j-3}&b_{j-2}&\dots&b_{2j-5}&b_{2j-4}&b_{2j-3}\\
                \vdots&\vdots&\vdots&\vdots&\vdots&\vdots&\ddots&\vdots&\vdots&\vdots&\ddots&\vdots&\vdots\\
                 0 & 0 & 0 & 0 & 0 & 0 &\dots&   0   &   0   &\dots&0&   c_0  &   c_1  \\
                 0 & 0 & 0 & 0 & 0 & 0 &\dots&  b_0  &  b_1  &\dots&b_{j-2}& b_{j-1}&  b_{j}
        \end{vmatrix}~.
\end{equation}
Here we set $b_k:=0$ for $k>r$, and $c_j:=0$ for $j>m$.
\begin{lemma}\label{lem.rat.1}
For any $j=1,2,\dots$,
\begin{equation}\label{Omega}
\Omega_{2j}(h,g)=c_0^{2j}\Delta_j(R),
\end{equation}
where $\Delta_j(R)$ are the Hurwitz determinants of the function $R$ defined in~\eqref{Gurwitz.minors.rat}.
\end{lemma}
\begin{proof}
First interchange the rows of the determinant~\eqref{nabla} to obtain
\begin{equation}\label{nabla.prof.1}
\Omega_{2j}(h,g)=
        \begin{vmatrix}
                c_0&c_1&c_2&c_3&\cdots&c_{j-1}&  c_j  &\cdots&c_{2j-2}&c_{2j-1}\\
                 0 & 0 &c_0&c_1&\cdots&c_{j-3}&c_{j-2}&\cdots&c_{2j-4}&c_{2j-3}\\
                \vdots&\vdots&\vdots&\vdots&\ddots&\vdots&\vdots&\ddots&\vdots&\vdots\\
                 0 & 0 & 0 & 0 &\cdots&\qquad   0  \qquad\surd &   0   &\dots&   c_0  &   c_1  \\
                 0 & 0 & 0 & 0 &\cdots&\qquad  b_0 \qquad\surd &  b_1  &\dots& b_{j-1}&   b_j  \\
                \vdots&\vdots&\vdots&\vdots&\ddots&\vdots&\vdots&\ddots&\vdots&\vdots\\
                 0 &b_0&b_1&b_2&\cdots&b_{j-2}&b_{j-1}&\dots&b_{2j-3}&b_{2j-2}\\
                b_0&b_1&b_2&b_3&\cdots&b_{j-1}&  b_j  &\dots&b_{2j-2}&b_{2j-1}
        \end{vmatrix}
\end{equation}
This does not change the sign of~$\Omega_{2j}(h,g)$. In fact, lower the $j$th and $(j+1)$st
rows\footnote{They are marked by $\surd$ in~\eqref{nabla.prof.1}.} to their initial positions
in~\eqref{nabla}. This will require an \textit{even} number of transpositions. The next pair
of rows will then meet, the lowering operation will be applied to them, and so on ($j-1$ times) until
we obtain the initial determinant~\eqref{nabla}.

Now we rearrange the columns of the determinant~\eqref{nabla.prof.1} in the
following way: $1$st, $3$d,\ldots,$(2j-1)$st, $2$nd, $4$th,\ldots, $2j$th.
To move the $3$d column to the second place, we need one transposition.
To move the $5$th column to the third place, we need two transpositions,
and so on. At last, to move the $(2j-1)$st column to the $j$th place,
we need $(j-1)$ transpositions. During all those transpositions the resulting
determinant will change its sign $\sum\limits_{i=1}^{j-1}i={\dfrac{j(j-1)}2}$ times.
Next we interchange the $(j+1)$st row with the $2j$th row, the $(j+2)$nd row with the $(2j-1)$st row,
and so on. This will also require $\dfrac{j(j-1)}2$ transpositions. Thus, the resulting determinant
has the same sign as the initial determinant, so we have
\begin{equation}\label{nabla.prof.2}
\Omega_{2j}(h,g)=
        \begin{vmatrix}
                c_0&c_2&c_4&\cdots&c_{2j-2}&c_1&c_3&c_5&\cdots&c_{2j-1}\\
                 0 &c_0&c_2&\cdots&c_{2j-4}& 0 &c_1&c_3&\cdots&c_{2j-3}\\
                 0 & 0 &c_0&\cdots&c_{2j-6}& 0 & 0 &c_1&\cdots&c_{2j-5}\\
                \vdots&\vdots&\vdots&\ddots&\vdots&\vdots&\vdots&\vdots&\ddots&\vdots\\
                 0 & 0 & 0 &\cdots&   c_0  & 0 & 0 & 0 &\cdots&   c_1  \\
                b_0&b_2&b_4&\cdots&b_{2j-2}&b_1&b_3&b_5&\cdots&b_{2j-1}\\
                 0 &b_1&b_3&\cdots&b_{2j-3}&b_0&b_2&b_4&\cdots&b_{2j-2}\\
                 0 &b_0&b_2&\cdots&b_{2j-4}& 0 &b_1&b_3&\cdots&b_{2j-3}\\
                 0 & 0 &b_1&\cdots&b_{2j-5}& 0 &b_0&b_2&\cdots&b_{2j-4}\\
                \vdots&\vdots&\vdots&\ddots&\vdots&\vdots&\vdots&\vdots&\ddots&\vdots\\
                 0 & 0 & 0 &\cdots&b_{j-1} & 0 & 0 & 0 &\cdots&b_j
        \end{vmatrix}
\end{equation}

Now from each $(j+2i-1)$st row,
$i=1,2,\dots,\left[\dfrac{j+1}2\right]$, we subtract rows $i$th, $(i+1)$st, $\dots$, $j$th multiplied by $t_0$, $t_2$,
$\dots$,$t_{2(j-i)-2}$, respectively. Then, from each $(j+2i)$th row, $i=1,2,\dots,\left[\dfrac j2\right]$, subtract
rows $(i+1)$st, $(i+2)$nd, $\dots$, $j$th multiplied by $t_1$, $t_3$, $\dots$,$t_{2(j-i)-1}$, respectively. As a result,
we have a determinant, which can be represented in a block form:
\begin{equation}\label{nabla.prof.3}
\Omega_{2j}(h,g)=
\begin{vmatrix}
     C_0  &C_1\\
     \widetilde{A}& \widehat{A}
\end{vmatrix}
\end{equation}
where the upper triangular matrices $C_0$ and $C_1$ have the forms
\begin{equation*}
C_0=
\begin{pmatrix}
             c_0&c_2&c_4&\cdots&c_{2j-2}\\
             0 &c_0&c_2&\cdots&c_{2j-4}\\
             0 & 0 &c_0&\cdots&c_{2j-6}\\
             \vdots&\vdots&\vdots&\ddots&\vdots\\
             0 & 0 & 0 &\cdots&   c_0  \\
\end{pmatrix}\,,\quad
C_1=
\begin{pmatrix}
             c_1&c_3&c_5&\cdots&c_{2j-1}\\
              0 &c_1&c_3&\cdots&c_{2j-3}\\
              0 & 0 &c_1&\cdots&c_{2j-5}\\
             \vdots&\vdots&\vdots&\ddots&\vdots\\
              0 & 0 & 0 &\cdots&   c_1  \\
\end{pmatrix}\,.
\end{equation*}
Note that $|C_0|=c_0^j>0$.

In order to describe the matrices $\widetilde{A}$ and $\widehat{A}$, we establish a connection between
the coefficients $b_i$s, $c_i$s and $t_i$s. Multiplying~\eqref{rat.func.1} by the denominator
and equating coefficients, we get
\begin{equation}\label{nabla.prof.4}
b_k=\sum_{i=0}^{k}c_{k-i}t_i,\quad k=0,1,2,\ldots
\end{equation}
Taking into account those formul\ae, we obtain for the entries of the matrix $\widetilde{A}$ to have the form:
\begin{equation}\label{nabla.prof.5}
\begin{aligned}
&\widetilde{a}_{2i-1,\/k}=
\begin{cases}
0\ \qquad\qquad\qquad\qquad\qquad\qquad\qquad\qquad\qquad\qquad\qquad\,\,\text{if}\ k<i+1,\\
\displaystyle
b_{2(k-i)}-\sum_{q=0}^{k-i}c_{2(k-i-q)}t_{2q}=\sum_{q=1}^{k-i}c_{2(k-i-q)+1}t_{2q-1}
\ \qquad\text{if}\ k\geqslant i+1,
\end{cases}\\
&\qquad\qquad\text{where}\ i=1,2,\dots,\text{\footnotesize{$\left[\frac{j+1}{2}\right]$}};\\
\\
&\widetilde{a}_{2i,\/k}=
\begin{cases}
0\ \qquad\qquad\qquad\qquad\qquad\qquad\qquad\qquad\qquad\qquad\qquad\qquad\,\,\text{if}\ k<i+1,\\
\displaystyle
b_{2(k-i)-1}-\sum_{q=1}^{k-i}c_{2(k-i-q)}t_{2q-1}=\sum_{q=0}^{k-i-1}c_{2(k-i-q)-1}t_{2q}
\ \qquad\text{if}\ k\geqslant i+1,
\end{cases}\\
&\qquad\quad\text{where}\
i=1,2,\dots,\text{\footnotesize{$\left[\,\frac{j}{2}\,\right]$}}.
\end{aligned}
\end{equation}
From these formul\ae~it follows that the matrix $\widetilde{A}$ can be represented as a product of
two matrices:
\begin{equation}\label{nabla.prof.6}
\widetilde{A}=
\begin{pmatrix}
              0&t_1&t_3&\cdots&t_{2j-3}\\
              0&t_0&t_2&\cdots&t_{2j-4}\\
              0& 0 &t_1&\cdots&t_{2j-5}\\
             \vdots&\vdots&\vdots&\ddots&\vdots\\
              0& 0 & 0 &\cdots&t_{j-2}\\
\end{pmatrix}
\begin{pmatrix}
             c_1&c_3&c_5&\cdots&c_{2j-1}\\
              0 &c_1&c_3&\cdots&c_{2j-3}\\
              0 & 0 &c_1&\cdots&c_{2j-5}\\
             \vdots&\vdots&\vdots&\ddots&\vdots\\
              0 & 0 & 0 &\cdots&   c_1  \\
\end{pmatrix}
\end{equation}
Analogously, by~\eqref{nabla.prof.4}, for the entries of the matrix $\widehat{A}$, we have
\begin{equation*}
\begin{aligned}
&\widehat{a}_{2i-1,\/k}=
\begin{cases}
0\ \qquad\qquad\qquad\qquad\qquad\qquad\qquad\qquad\qquad\qquad\qquad\quad\,\text{if}\ k<i,\\
\displaystyle
b_{2(k-i)+1}-\sum_{q=0}^{k-i}c_{2(k-i-q)+1}t_{2q}=\sum_{q=0}^{k-i}c_{2(k-i-q)}t_{2q+1}\ \qquad
\text{if}\ k\geqslant i,
\end{cases}\\
&\qquad\qquad\text{where}\ i=1,2,\dots,\text{\footnotesize{$\left[\frac{j+1}{2}\right]$}};\\
\end{aligned}
\end{equation*}
\begin{equation*}
\begin{aligned}
&\widehat{a}_{2i,\/k}=
\begin{cases}
0\ \qquad\qquad\qquad\qquad\qquad\qquad\qquad\qquad\qquad\qquad\qquad\,\,\text{if}\ k<i,\\
c_0t_0\ \qquad\qquad\qquad\qquad\qquad\qquad\qquad\qquad\qquad\qquad\quad\,\text{if}\ k=i,\\
\displaystyle
b_{2(k-i)}-\sum_{q=1}^{k-i}c_{2(k-i-q)+1}t_{2q-1}=\sum_{q=0}^{k-i}c_{2(k-i-q)}t_{2q}\
\qquad\text{if}\ k>i,
\end{cases}\\
&\qquad\quad\text{where}\
i=1,2,\dots,\text{\footnotesize{$\left[\,\frac{j}{2}\,\right]$}}.
\end{aligned}
\end{equation*}
Therefore, the matrix~$\widehat{A}$ can be represented as follows
\begin{equation}\label{nabla.prof.7}
\widehat{A}=
\begin{pmatrix}
              t_1&t_3&t_5&\cdots&t_{2j-1}\\
              t_0&t_2&t_4&\cdots&t_{2j-2}\\
               0 &t_1&t_3&\cdots&t_{2j-3}\\
             \vdots&\vdots&\vdots&\ddots&\vdots\\
               0 & 0 & 0 &\cdots&t_{j}\\
\end{pmatrix}
\begin{pmatrix}
             c_0&c_2&c_4&\cdots&c_{2j-2}\\
              0 &c_0&c_2&\cdots&c_{2j-4}\\
              0 & 0 &c_0&\cdots&c_{2j-6}\\
             \vdots&\vdots&\vdots&\ddots&\vdots\\
              0 & 0 & 0 &\cdots&   c_0  \\
\end{pmatrix}=:\Lambda_jC_0
\end{equation}
It is clear that $\Lambda_j$ is the $j\times j$ leading principal submatrix of the matrix $\mathcal{H}(R)$ defined in~\eqref{Hurwitz.matrix}.
Taking~\eqref{nabla.prof.7} into account, we multiply both parts of the equality~\eqref{nabla.prof.3} by the following
determinant
\begin{equation*}
\begin{vmatrix}
              E_j&0\\
               0 &C_0^{-1}
\end{vmatrix}=c_0^{-j}>0,
\end{equation*}
where $E_j$ is the $j\times j$ identity matrix, to obtain
\begin{equation*}\label{nabla.prof.8}
c_0^{-j}\Omega_{2j}(h,g)=
\begin{vmatrix}
     C_0  &C_1C_0^{-1}\\
     \widetilde{A}&\Lambda_j
\end{vmatrix},
\end{equation*}
or in the entry-wise from:
\begin{equation}\label{nabla.prof.9}
c_0^{-j}\Omega_{2j}(h,g)=
        \begin{vmatrix}
                c_0&c_2&c_4&\cdots&c_{2j-2}&\dfrac{c_1}{c_0}&\cdots&\cdots&\cdots&\cdots\\
                 0 &c_0&c_2&\cdots&c_{2j-4}& 0 &\dfrac{c_1}{c_0}&\cdots&\cdots&\cdots\\
                 0 & 0 &c_0&\cdots&c_{2j-6}& 0 & 0 &\dfrac{c_1}{c_0}&\cdots&\cdots\\
                \vdots&\vdots&\vdots&\ddots&\vdots&\vdots&\vdots&\vdots&\vdots&\vdots\\
                 0 & 0 & 0 &\cdots&   c_0  & 0 & 0 & 0 &\cdots&\dfrac{c_1}{c_0}\\
                 0 &c_1t_1&c_3t_1+c_1t_3&\cdots&\cdots&t_1&t_3&t_5&\cdots&t_{2j-1}\\
                 0 &c_1t_0&c_3t_0+c_1t_2&\cdots&\cdots&t_0&t_2&t_4&\cdots&t_{2j-2}\\
                 0 &   0  &c_1t_1&\cdots&\cdots& 0 &t_1&t_3&\cdots&t_{2j-3}\\
               \vdots&\vdots&\vdots&\ddots&\ddots&\vdots&\vdots&\vdots&\ddots&\vdots\\
                 0 &   0  &   0  &\cdots&\cdots& 0 & 0 & 0 &\cdots&t_{j}
        \end{vmatrix}
\end{equation}
The formul\ae~\eqref{nabla.prof.5}--\/\eqref{nabla.prof.9} show that all the columns of the
matrix~$\widetilde{A}$ are linear combinations of the columns of the
matrix~$\Lambda_j$. Consequently, we are able to eliminate the entries in the lower left corner of
the determinant~\eqref{nabla.prof.9}. To do this, from the $2$nd column we subtract the $(j+1)$st
column multiplied by $c_1$. Then from the $3$т column we subtract the $(j+1)$st column multiplied
by $c_1$ and the $(j+2)$nd column multiplied by $c_3$, and so on. As a result, we obtain the
determinant $c_0^{-j}\Omega_{2j}(h,g)$ to have the following form
\begin{equation*}\label{nabla.prof.10}
c_0^{-j}\Omega_{2j}(h,g)=
        \begin{vmatrix}
                c_0&\cdots&\cdots&\cdots&\cdots&\dfrac{c_1}{c_0}&\cdots&\cdots&\cdots&\cdots\\
                 0 &c_0&\cdots&\cdots&\cdots& 0 &\dfrac{c_1}{c_0}&\cdots&\cdots&\cdots\\
                 0 & 0 &c_0&\hdotsfor{2}& 0 & 0 &\dfrac{c_1}{c_0}&\hdotsfor{2}\\
                \cdots&\cdots&\cdots&\cdots&\cdots&\cdots&\cdots&\cdots&\cdots&\cdots\\
                 0 & 0 & 0 &\cdots&c_0& 0 & 0 & 0 &\cdots&\dfrac{c_1}{c_0}\\
                 0 & 0 & 0 &\cdots& 0 &t_1&t_3&t_5&\cdots&t_{2j-1}\\
                 0 & 0 & 0 &\cdots& 0 &t_0&t_2&t_4&\cdots&t_{2j-2}\\
                 0 & 0 & 0 &\cdots& 0 & 0 &t_1&t_3&\cdots&t_{2j-3}\\
               \cdots&\cdots&\cdots&\cdots&\cdots&\cdots&\cdots&\cdots&\cdots&\cdots\\
                 0 & 0 & 0 &\cdots& 0 & 0 & 0 & 0 &\cdots&t_j\\
        \end{vmatrix}=c_0^j|\Lambda_j|
\end{equation*}

Since $|\Lambda_j|=\Delta_j(R)$, we get
\begin{equation*}
\Omega_{2j}(h,g)=c_0^{2j}\Delta_j(R),
\end{equation*}
as required.
\end{proof}

Now Lemma~\ref{lem.rat.1} and Theorem~\ref{Hurvitz.analog} imply the following
criterion of Hurwitzness of real rational functions in terms of the coefficients
of their numerator and denominator.
\begin{theorem}\label{Hurvitz.analog.2}
A real rational function $R$ defined in~\eqref{rat.func.1}--\eqref{rat.func.2} is Hurwitz
if and only if the inequalities
\begin{equation*}
\Omega_2(h,g)>0,\ \Omega_4(h,g)>0,\dots,\ \Omega_{2n}(h,g)>0
\end{equation*}
hold.
\end{theorem}

From~\eqref{zeros} and~\eqref{Omega} it also follows that
\begin{equation*}
\Omega_{2n+2}(h,g)=\Omega_{2n+4}(h,g)=\Omega_{2n+6}(h,g)=\cdots=0.
\end{equation*}

Lemma~\ref{lem.rat.1} and the formul\ae~\eqref{Hurvitz.dets} imply the following corollary.
\begin{corol}
Let the polynomials $h$ and $g$ be defined in~\eqref{rat.func.2}. Then the Hurwitz minors
of the polynomial $P(z)=(-1)^mh(z)g(-z)$ satisfy the equalities:
\begin{equation*}
\Delta_j(P)=\Omega_{2j}(h,g),\qquad j=1,2\ldots,n,
\end{equation*}
where $n=\deg P=\deg h+\deg g$.
\end{corol}

Consider now the rational function
\begin{equation}\label{function.F}
F(z)=\dfrac{(-1)^n}{R(-z)},
\end{equation}
where $R$ is defined in~\eqref{rat.func.1}--\/\eqref{rat.func.2}, and $n=r+m$ is order of $R$.

It is clear that the function $F$ is Hurwitz if and only if the function $R$ is Hurwitz. This fact can be verified by
the following relationship between the Hurwitz minors of the functions $F$ and $R$.
\begin{theorem}
Let the function $R$ be defined in~\eqref{rat.func.1}--\/\eqref{rat.func.2}, and let the function $F$ be
defined in~\eqref{function.F}. Then
\begin{equation}\label{Deltas}
\Delta_j(R)=t_0^{2j}\Delta_j(F),\qquad j=1,2,\ldots
\end{equation}
\end{theorem}
\begin{proof}
Let $F(z)=\dfrac{f(z)}{q(z)}$, where $f(z)=(-1)^mg(-z)$, and $q(z)=(-1)^rh(-z)$. Consider the polynomial $Q(z)=(-1)^rq(-z)f(z)$, which is analogous to~\eqref{rat.func.3}. We have
\begin{equation*}
Q(z)=(-1)^rq(-z)f(z)=(-1)^r(-1)^rh(z)(-1)^{m}g(-z)=(-1)^mh(z)g(-z)=P(z).
\end{equation*}
Now the formul\ae~\eqref{Hurvitz.dets} imply
\begin{equation*}
c_0^{2j}\Delta_j(R)=\Delta_j(P)=\Delta_j(Q)=b_0^{2j}\Delta_j(F),\qquad j=1,2,\ldots,
\end{equation*}
that is exactly~\eqref{Deltas}, since $t_0=\dfrac{b_0}{c_0}$.
\end{proof}

\begin{corol}
Let the polynomials $h$ and $g$ are defined in~\eqref{rat.func.2}. Then
\begin{equation*}
\Omega_{2j}(h,g)=\Omega_{2j}(f,q),\qquad j=1,2,\ldots,
\end{equation*}
where $f(z)=(-1)^mg(-z)$, and $q(z)=(-1)^rh(-z)$.
\end{corol}

Finally, note that if one of the polynomials $h$ and $g$ is a constant, the determinants $\Omega_{2j}(h,g)$ become Hurwitz determinants
(up to a constant factor) of the polynomial, which is not a constant. This follows from the formul\ae~\eqref{Hurvitz.dets}.

\section*{Acknowledgments}
The authors are grateful to Professor Olga Holtz for helpful discussions.

\end{document}